\documentclass{amsart}
\usepackage[english]{babel}
 
\usepackage[a4paper,hmargin={3.2cm,3.2cm},vmargin={3.2cm,3.2cm},heightrounded, marginparwidth=3.1cm, marginparsep=0.05cm]{geometry}

\usepackage[utf8]{inputenc}

\usepackage[colorinlistoftodos]{todonotes} 

\newcommand{\wt}[1]{\
  \ifhmode\todo{#1}\else\vadjust{\todo{#1}}\fi}
\newcommand{\ls}[1]{\
  \ifhmode\todo[color=blue!40]{#1}\else\vadjust{\todo[color=blue!40]{#1}}\fi}

\usepackage[all]{xy}
\usepackage{amsbsy,mathrsfs,latexsym,amssymb,mathbbol,stmaryrd}
\usepackage[mathcal]{euscript}
\usepackage{amscd,amsthm}

\usepackage{hyperref} 
\usepackage{verbatim} 

\input xy
\xyoption{all} \xyoption{arc}
\usepackage[curve]{xypic}

\usepackage{geometry}
\usepackage{enumitem}
\usepackage[all]{xy}
\usepackage{mathtools}

\newcommand{\ca}{\mathcal{A}}

\newcommand{\cx}{\mathcal{X}}

\newcommand{\cd}{\mathcal{D}}

\newcommand{\ce}{\mathcal{E}}

\newcommand{\cs}{\mathcal{S}}
\newcommand{\ct}{\mathcal{T}}

\newcommand{\Pos}{\mathcal{P}\hspace*{-0.15mm}os}
\newcommand{\Posl}{\mathcal{POS}}

\newcommand{\Top}{\mathcal{T}\hspace*{-0.45mm}op}
\newcommand{\Set}{\mathcal{S}et}
\newcommand{\Ord}{\mathcal{O}rd}

\newcommand\Frm{\mathcal{F}\hspace*{-0.1mm}rm}
\newcommand\SLat{\mathcal{SL}\hspace*{-0.1mm}at}
\newcommand{\OVec}{\mathcal{OV}\hspace*{-0.1mm}ec}

\newcommand{\POVec}{\mathcal{POV}\hspace*{-0.1mm}ec}

\newcommand{\id}{\mathrm{id}}
\newcommand{\op}{\mathrm{op}}
\newcommand{\ob}{\mathrm{ob}}

\newcommand{\ot}{\cdot}

\theoremstyle{plain}
\newtheorem{theorem}{Theorem}[section]
\newtheorem{proposition}[theorem]{Proposition}
\newtheorem{corollary}[theorem]{Corollary}
\newtheorem{lemma}[theorem]{Lemma}
\theoremstyle{definition}
\newtheorem{definition}[theorem]{Definition}

\newtheorem{example}[theorem]{Example}
\newtheorem{examples}[theorem]{Examples}

\newtheorem{remark}[theorem]{Remark}
\newtheorem{remarks}[theorem]{Remarks}

\newtheorem{oproblem}[theorem]{Open Problem}

\title{Order-enriched solid functors}

\author{Lurdes Sousa}

\address{CMUC, University of Coimbra, 3001-501
  Coimbra, Portugal and\newline\indent Polytechnic Institute of Viseu,
  3504-510 Viseu, Portugal}

\email{sousa@estv.ipv.pt}

\author{Walter Tholen}

\address{York
  University, Toronto ON, M3J 1P3, Canada}

\email{tholen@mathstat.yorku.ca}

\thanks{This work was partially supported by the Centre for Mathematics of the 
University of Coimbra -- UID/MAT/00324/2019, funded by the Portuguese 
Government through FCT/MEC and co-funded by the European Regional 
Development Fund through the Partnership Agreement PT2020.
  The second author acknowledges partial financial support under the Discovery Grants
  Program by the Natural Sciences and Engineering Research Council of Canada.}


\subjclass[2010]{}

\date{\today}

\begin{document}

\maketitle

\centerline{\em In memory of V\v{e}ra Trnkov\'{a}}

\begin{abstract} Order-enriched solid functors, as presented in this paper in two versions, enjoy many of the strong properties of their ordinary counterparts, including the transfer of the existence of weighted (co)limits from their codomains to their domains. The ordinary version of the notion first appeared in Trnkov\'a's work on automata theory of the 1970s and was subsequently studied by others under various names, before being put into a general enriched context by Anghel. Our focus in this paper is on differentiating the order-enriched notion from the ordinary one, mostly in terms of the functor's behaviour with respect to specific weighted (co)limits, and on the presentation of examples, which include functors of general varieties of ordered algebras and special ones, such as ordered vector spaces.
\end{abstract}

\section{Introduction}\label{sec-int}
Inspired by \v{C}ech's book \cite{Cech} and Hu\v{s}ek's article \cite{Husek}, in her work \cite{Tr75} on \emph{Automata and Categories} V\v{e}ra Trnkov\'a defined a concrete category $\ca$ (which therefore comes with a faithful functor $|\text{-}|:\ca\to\Set$) to \emph{admit weak inductive generation} if, for every possibly large (!) family $(D_i)_{i\in I}$ of $\ca$-objects, equipped with maps $\xi_i: |D_i| \to X\,(i\in I)$ into a given set $X$, there exist an $\ca$-object $A$ and a map $q:X\to |A|$ such that 
\begin{itemize}
\item[1.] all maps $q\cdot \xi_i: |D_i|\to |A|$ underly $\ca$-morphisms $D_i\to A$, and 

\item[2.] the pair $(q,A)$ is universal with this property, \emph{i.e.}, for every map $f:X\to |B|$ with $B$ in $\ca$, such that all maps $f\cdot \xi_i:|D_i|\to |B|$ underly $\ca$-morphisms $D_i\to B$, there is a unique $\ca$-morphism $t:A\to B$ with $ |t|\cdot q = f$. 
\end{itemize}
When $q$ may always be chosen to be the identity map (so that $A$ has underlying set $X$), this gives precisely the notion of concrete \emph{topological category} or, when one trades $\Set$ for any ``base category" $\cx$, of (faithful) \emph{topological functor} $P:\ca\to\cx$ (see, for example, \cite{AHS, HST2014}). This, as it turns out, self-dual notion was first introduced (in dual form and under a different name) in Br\"{u}mmer's thesis \cite{Brummer}; Trnkov\'{a}, not aware of \cite{Brummer}, calls $\ca$ to \emph{admit inductive generation} in this case.

Other precursors to the notion of topological category or functor (first just over $\Set$, but then over any category $\cx$), such as \cite{Taylor, Antoine, Roberts, Wyler71a, Wyler71b, Manes, Shukla, H72, Wischnewsky, Hong}, limited the concept of inductive generation and its dualization to the consideration of small families, or even singleton-families (thus essentially considering Grothendieck's \emph{bifibrations}), and then imposed smallness conditions to effectively enable inductive generation for large families of data, such as asking each fibre of the given functor (\emph{i.e.} each category of $\ca$-objects with fixed underlying $\cx$-object) to form a small complete lattice, as it was done by Wyler in his milestone papers \cite{Wyler71a, Wyler71b}. While credit for having elegantly introduced topologicity of a concrete category using large families is due to \cite{Brummer}, this approach resonated with a wider audience only after the appearance of Herrlich's important article \cite{Herrlich74}, with its notions quickly expanded upon in other papers, such as \cite{T76,W78,T78}.

For small families and, more generally, for small cocones, the concept of \emph{weak} inductive generation was, without Trnkov\'{a}'s knowledge at the time, considered earlier by Hoffmann in \cite{H72}, under a different name and in rather cumbersome notation, and it appeared in published form only later, in \cite{H76}. Afterwards, unaware of Trnkov\'{a}'s notion, the authors of \cite{T79, TW79, STWW80, T80} undertook
a systematic and coherent study of the $\cx$-based categories $\ca$ admitting weak inductive generation, showing their usefulness in the categorical investigation of a wide range of mathematical structures. These papers reconcile many themes studied earlier in the the more restrictive context of topological categories and therefore call the functors involved (presenting the categories $\ca$ as concrete over $\cx$) \emph{semi-topological}, a term that had been used somewhat hiddenly in \cite{H76}. Even though every such functor may be presented as the composite of a full reflective embedding followed by a topological functor, the occurrence of these functors is by no means restricted to the realm of topology. Therefore, on Herrlich's suggestion, they were renamed as \emph{solid} later on, a term adopted in \cite{BT90, AHS} and used henceforth by others, for example in \cite{Sousa95}.

In his thesis \cite{Anghel}, published and extended in \cite{Anghel1, Anghel2}, Anghel takes the study of the (then) semi-topological functors comprehensively to the level of enriched category theory \cite{Kelly82}. However, in order to do so, he needed to utilize the full range of the theory and often to impose additional conditions on the categories at issue, making it somewhat hard for the non-expert to apply his results. The purpose of this article is therefore to present a largely self-contained theory of solid functors in the easily presentable context of {\em order-enriched} categories and functors and their applications, that is: in an environment that has gained considerable attention in recent years; see, for example, \cite{AS17, ASV15, CS11, CS17, Sousa17}.\footnote{For the purpose of consistency with these papers, but at the price of divergence from other works (such as \cite{HST2014, Fritz2017}), in this paper we understand ``order" to mean what is generally referred to as ``partial order". But we stress the fact that the theory presented here carries through smoothly when ``order" means just ``preorder" in more common parlance, perhaps even more so than in the partially ordered context. In fact, many general constructions lead from partially ordered sets just to preordered sets, which at the end have to be subjected to the reflector to enforce separation ( = anti-symmetry), as demonstrated also by some of the examples presented in Section \ref{sec examples}.}  Explicitly then, the hom-sets of our categories come equipped with a partial order that is preserved by the composition of the category and by the hom-maps of any functors departing from them, thus providing also an elementary 2-categorical context in which all 2-cells are given by order. However, even in this very simplified context one quickly arrives at subtleties that hinder a seamless transition of notions and results from the ordinary to the enriched context.

Therefore, in Section \ref{sec strongly solid} we first present a notion of solidity for ordered functors, called {\em strongly order-solid}\footnote{In this paper the easily defined strongly order-solid functors appear before the more natural, but also slightly more complex, notion of order-solid functor, since we are not aware of examples of the latter type of functors not already covered by the former.}, which on first sight seems to add only a minor order-related condition to the ordinary notion. Nevertheless, it captures an extensive list of relevant examples, some of which appear in Section \ref{sec examples}. It then turns out that the seemingly mild additional condition which makes ordinarily solid functors strongly  order-solid already guarantees that they become solid as order-enriched functors in Anghel's sense \cite{Anghel1}, called {\em order-solid} here. We present these functors in Section \ref{sec order solid} without assuming the reader's familiarity with \cite{Kelly82}. While strongly  order-solid functors are easily seen to be order-solid, the converse question, whether every order-solid functor is strongly order-solid, is still open.

A central goal of the paper is the characterization of strongly  order-solid and of order-solid functors in terms of their behaviour vis-a-vis weighted limits and colimits. In Theorem \ref{thm characterization} we characterize strongly  order-solid functors using inserters, and in Theorem \ref{strong vs order} we state that they ``lift" the existence of weighted (co)limits for diagrams of any given shape. We study the behaviour of order-solid functors on weighted colimits in Section \ref{sec colim} and characterize order-solid functors when the ``base" category is tensored (Theorem \ref{thm tensored}). The list of examples in Section \ref{sec examples} culminates in a theorem on categories of general ordered algebras; Theorem \ref{alg fun solid} asserts that algebraic functors between them are always strongly  order-solid as soon as they admit free algebras over every ordered set. The category of ordered vector spaces, considered as an ordered category via the positive cones of its objects, falls outside the scope of this theorem, but its positive-cone functor to the category of partially ordered sets is still strongly order-solid. When considered as a discretely ordered category, it serves as a resource to demonstrate that certain conditions of our characterization theorems are essential.

\section{Strongly order-solid functors}\label{sec strongly solid}

We generally assume our categories and functors to be enriched in the Cartesian closed category $\Pos$ of (partially) ordered sets and their monotone (= order-preserving) maps and simply call them {\em ordered}\footnote{As mentioned in the Introduction, and as will become apparent in Section \ref{sec examples}, for many purposes we may alternatively work with the Cartesian closed category $\Ord$ of preordered sets and their monotone maps.}. Hence, the hom-sets of an ordered category $\ca$ carry an order which is preserved by composition with morphisms from either side, and the hom-maps of an ordered functor $P:\ca\to\cx$  preserve the order as well. In accordance with \cite{Kelly82}, whenever necessary for clarity, we write $\ca_o$ for the underlying ordinary category of an ordered category $\ca$, and likewise for ordered functors. 

Recall that the ordered functor $P:\ca\to\cx$ has a left adjoint $F:\cx\to\ca$ in the order-enriched sense if  there are order-isomorphisms
$$\ca(FX,B)\cong\cx(X,PB)$$
that are natural in $X\in{\rm ob}\cx$ and $B\in{\rm ob}\ca$; we call $P$ {\em order-right adjoint} in this case. For that to happen it suffices that, for every $\cx$-object $X$, one finds a (tacitly chosen) $P$-universal arrow $e:X\to PA$ with $A\in{\rm ob}\ca$ \cite{MacLane71} which has the additional property of being {\em order}-$P$-{\em epi(morphi)c}, that is: whenever $Pr\cdot e \leq Ps\cdot e$ for any morphisms $r,s:A\to B$ in $\ca$, then $r\leq s$; equivalently, the ordered functor $P$ is right adjoint in the ordinary sense such that all adjunction units are order-$P$-epic.

Given a (potentially large\footnote{ That is: the size of the indexing system $I$ may be as large as the size of the class of all morphisms of $\ca$.}) family $D=(D_i)_{i\in I}$ of objects in $\ca$, we consider the (potentially very large\footnote{We use the term ``very large" informally, to refer to collections of (potentially proper) classes, called conglomerates in \cite{AHS}. A formalization of the term does not seem to be justified in this paper, since one may, of course, avoid the formation of $D\downarrow \ca$ and $PD\downarrow\cx$ (the individual objects of which may already be large), but one will then have to accept universal quantification over these entities: see Definition \ref{def strongly solid}.}) category $D\downarrow\ca$ whose objects are pairs $(\alpha, A)$ with an $\ca$-object $A$ and a family $\alpha=(\alpha_i)_{i\in I}$ of $\ca$-morphisms $\alpha_i:D_i\to A$, shortly written as $\alpha:D\to A$; a morphism $t:(\alpha,A)\to(\beta,B)$ is given by an $\ca$-morphism $t:A\to B$ satisfying $t\cdot\alpha_i=\beta_i$ for all $i\in I$, shortly written as $t\cdot\alpha=\beta$. Of course, $D\downarrow A$ and, likewise, $PD\downarrow\cx$ inherit the order from $\ca$ and $\cx$, respectively, making both categories ordered, as well as the $P$-induced functor
$$P_D:(D\downarrow\ca)\to(PD\downarrow\cx),\quad (\alpha,A)\mapsto(P\alpha,PA).$$

\begin{definition}\label{def strongly solid} 
An ordered functor $P:\ca\to\cx$ is {\em strongly order-solid} if the functor $P_D$ is order-right adjoint for every family $D$ of $\ca$-objects. Equivalently, given any $D$, for every family $\xi:PD\to X$ in $\cx$ there is a (tacitly chosen) family $\alpha:D\to A$ in $\ca$ and an $\cx$-arrow $q:X\to PA$  such that
\begin{itemize} 
\item[1.] $(\alpha, A, q)$ is a $P${\em-extension of} $\xi$, that is: $P\alpha=q\cdot\xi$; 
\item[2.] $(\alpha, A, q)$ is {\em universal} with respect to property 1, that is: for every family $\beta:D\to B$ in $\ca$ and every $\cx$-arrow $f:X\to PB$ with $P\beta=f\cdot\xi$ one has a unique\footnote{Uniqueness comes for free in the presence of condition 3, but only so because here we understand ``ordered" to entail anti-symmetry. This observation applies analogously to many subsequent notions in this paper.} $\ca$-morphisms $t:A\to B$ with $t\cdot\alpha=\beta$ and $Pt\cdot q=f$; 
\item[3.] $q:X\to PA$ is order-$P$-epimorphic.
\end{itemize}
The three properties together make $(\alpha,A,q)$ a {\em strongly order-universal $P$-extension} of $\xi$.
\end{definition}

\begin{remarks}\label{first rem}
(1) Just as order-right adjoint functors are in particular right-adjoint ordinary functors, every strongly  order-solid functor is in particular solid in the ordinary sense and therefore faithful; see Lemma 3.2 of \cite{T79}, the proof of which uses a Cantor-type diagonal argument, as presented more generally in \cite{BT78}. But if $P$ is faithful, given a family $\xi:PD\to X$, any family $\beta:D\to B$ with $P\beta=f\cdot\xi$ is already determined by $B$ and $f:X\to PB$. Hence, the existence requirement of universal $P$-extensions for all $D$ and $\xi$ amounts precisely to Trnkov\'{a}'s admittance of weak inductive generation, as recorded at the beginning of the Introduction, to which we have only added the condition that all universal $P$-extensions be order-$P$-epic to make the ordered functor $P$ strongly order-solid.

To see that a strongly  order-solid functor as defined in \ref{def strongly solid} is faithful, one in fact does not need to resort to the above argument, as a stronger property may be shown easily: see Proposition \ref{prop order ff} below.

(2) Being in particular solid in the ordinary sense, a strongly  order-solid functor $P:\ca\to\cx$ certainly enjoys all the ``lifting properties" of solid functors, such as: if $\cx$ has all ordinary (co)limits (of diagrams of a specified shape), so does $\ca$ \cite{T79, AHS}; if $\cx$ is totally cocomplete (so that its Yoneda embedding has a left adjoint in the ordinary sense), so is $\ca$ \cite{T80}. 

(3) All fully faithful order-right adjoint functors are strongly order-solid, and so are composites of strongly  order-solid functors.
\end{remarks}

While we postpone the discussion of the behaviour of strongly  order-solid functors with respect to {\em weighted} (co)limits until Sections \ref{sec order solid} and \ref{sec colim}, here we consider one (easy, but important) type of weighted limit since it helps clarifying the relationship of the notions of strongly  order-solid functor and ordinarily solid functor.

\begin{definition}\label{def order ff}
(1) Recall that an {\em inserter} of a pair of morphisms $r,s:A\to B$ in an ordered category $\ca$ is a morphism $u:U\to A$ with $r\cdot u\leq s\cdot u$ that is universal with this property: any $v:V\to A$
 with $r\cdot v\leq s\cdot v$ factors (uniquely) as $v=u\cdot j$; moreover, $u$ is required to be {\em order mon(omorph)ic}, so that $u\cdot h\leq u\cdot k$ always implies $h\leq k$.

(2) An ordered functor $P:\ca\to\cx$ is {\em order-faithful} if $1_{PA}$ is order-$P$-epic for all objects $A$ in $\ca$, that is: $Pr\leq Ps$ for morphisms $r,s:A\to B$ in $\ca$ always implies $r\leq s$.
\end{definition}

\begin{proposition}\label{prop order ff}
A strongly  order-solid functor is solid in the ordinary sense, as well as order-right adjoint and order-faithful, and it preserves (any existing) inserters.
\end{proposition}

\begin{proof}
The first claim is obvious; see Remark \ref{first rem}(1). Since, in Definition \ref{def strongly solid}, families are allowed to be empty, order-right adjointness follows. It is standard to confirm the preservation of inserters (or any weighted limits) by order-right adjoint functors. So, only order-faithfulness of a strongly  order-solid functor $P$ needs to be shown here. But given $r,s:A\to B$ in $\ca$ with $Pr\leq Ps$, let $(a:A\to C, \,C, \,q:PA\to PC)$ be a universal $P$-extension of the singleton family $(1_{PA}:PA\to PA)$. Then $Pr$ and $Ps$ must both factor through $q$, so that for some $r',s':C\to B$ one has $Pr=Pr'\cdot q$ and $Ps=Ps'\cdot q$, as well as $r=r'\cdot a$ and $s=s'\cdot a$. Since $q$ is order-$P$-epic, $r'\leq s'$ follows, which implies $r\leq s$.
\end{proof}

\begin{proposition}\label{prop characterization}
Let $\ca$ have inserters. Then an ordered functor $P:\ca\to\cx$ is strongly  order-solid if, and only if, P is solid in the ordinary sense and order-faithful and preserves inserters.
\end{proposition}

\begin{proof}
After Proposition \ref{prop order ff}, only the ``if"-part needs proof. To this end, it suffices to show that, for the  given families $D$ and $\xi:PD\to X$ as in Definition \ref{def strongly solid}, the universal $P_o$-extension $(\alpha, A, q)$ with respect to the ordinary functor $P_o:\ca_o\to\cx_o$ serves also as a strongly order-universal $P$-extension, that is: $q:X\to PA$ 
is necessarily order-$P$-epic. Hence, assuming $Pr\cdot q\leq Ps\cdot q$ for $r,s:A\to B$ in $\ca$ we form the inserter $u:U\to A$ of the pair $r, s$ in $\ca$ which, by hypothesis, is preserved by $P$. So $q$ factors as $q=Pu\cdot f$, with $f:X\to PU$. Since $P(r\cdot\alpha)=Pr\cdot q\cdot\xi\leq Ps\cdot q\cdot\xi=P(s\cdot\alpha)$ and $P$ is order-faithful, $r\cdot\alpha\leq s\cdot\alpha$ follows. Consequently, the inserter $u$ makes the family $\alpha$ factor as $\alpha = u\cdot\beta$. Since $Pu\cdot f\cdot\xi=q\cdot\xi=Pu\cdot P\beta$ and $Pu$ (as an inserter) is monic in $\cx$, one obtains $f\cdot\xi=P\beta$ and therefore an $\ca$-morphism $t:A\to U$ with $Pt\cdot q= f$. From $P(u\cdot t)\cdot q=Pu\cdot f=q$ and $q$ being (ordinarily) $P$-epic, one derives $u\cdot t=1$. Since $r\cdot u\leq s\cdot u$, this finally implies $r\leq s$.
\end{proof}

\begin{theorem}\label{thm characterization}
Let $\cx$ have inserters. An ordered functor $P:\ca\to\cx$ is strongly  order-solid if and only if
\begin{itemize}
\item[{\em (a)}] $P$ is solid as an ordinary functor;
\item[{\em (b)}] $\ca$ has inserters and $P$ preserves them;
\item[{\em (c)}] $P$ is order-faithful.
\end{itemize}
\end{theorem}

\begin{proof}
That the conditions (a-c) are sufficient for $P$ to be strongly  order-solid has been confirmed in Proposition \ref{prop characterization}. Conversely, only the existence of inserters in $\ca$ still needs to be shown when $\cx$ has them and $P$ is strongly order-solid. 
To this end, for any morphisms 
$r,s:A\to B$ in $\ca$ we form the inserter $k:X\to PA$ of $Pr, Ps$ in $\cx$ and then consider the family $\xi$ of all pairs $(D,x)$ with $D\in {\rm ob}\ca$ and $x:PD\to X$ an $\cx$-morphism such that there is a (necessarily unique) $\ca$-morphism $a:D\to A$ with $Pa=k\cdot x$. (Note that, as an ordinarily solid functor, P is faithful.)
With $q:X\to PU$ forming a universal $P$-extension of $\xi$ we then see that $k$ must factor as $k=Pu\cdot q$ with $u: U\to A$ in $\ca$. Since $P(r\cdot u)\cdot q=Pr\cdot k\leq Ps\cdot k=P(s\cdot u)\cdot q$ and $q$ is order-$P$-epic, $r\cdot u\leq s\cdot u$ follows.

Furthermore, by the inserter property of $k$, any $v:V\to A$ in $\ca$ with $r\cdot v\leq s\cdot v$ produces a morphism $y:PV\to X$ with $k\cdot y=Pv$. This makes $(V,y)$ a member of the family $\xi$, which implies that there is an $\ca$-morphism $j:V\to U$ with $Pj=q\cdot y$. From $P(u\cdot j)=Pu\cdot q\cdot y=k\cdot y=Pv$ one obtains $u\cdot j= v$, as required. We note that the same argumentation may also be applied to $u$ in place of $v$; it produces morphisms $z: PU\to X$ and $t:U\to U$ with $k\cdot z=Pu$ and $Pt=q\cdot z$. Since $k$ is monic, from $k\cdot z\cdot q=k$ one first obtains $z\cdot q=1_X$, and then $Pt\cdot q=q\cdot z\cdot q=q$ forces $t=1_U$ since $q$ is $P$-epic. Consequently, $q\cdot z=1_{PU}$, so that $q$ and $z$ must be isomorphisms in $\cx$.

It remains to be shown that $u$ is order-monic. If $u\cdot c\leq u\cdot d$ with $c,d:C\to U$ in $\ca$, applying $P$ to the inequality we first obtain $k\cdot z\cdot Pc\leq k\cdot z\cdot Pd$ and then $Pc\leq Pd$, since $k$ is order-monic and $z$ an isomorphism. As $P$ is order-faithful, $c\leq d$ follows.
\end{proof}

We suppose that the existence assumptions regarding inserters are essential in Proposition \ref{prop characterization} and Theorem \ref{thm characterization} but have not been able yet to confirm this conjecture. However, preservation of inserters is: in Example \ref{OVec} we exhibit a solid and order-faithful functor $P:\ca\to\cx$ (thus satisfying conditions (a) and (c) of the above theorem), with both $\cx$ and $\ca$ having inserters, but with $P$ failing to preserve them; in particular, $P$ fails to be order-right adjoint and, {\em a fortiori}, strongly order-solid. This still leaves open the following question:

\begin{oproblem}
Is a (ordinarily) solid, order-right adjoint and order-faithful functor $P:\ca\to\cx$ strongly order-solid? Equivalently, when $\cx$ has inserters, do these conditions on $P$ imply the existence of inserters in $\ca$?
\end{oproblem}

\begin{remark}\label{remark E-char}
We recall from \cite{T79} (see Theorem 1.2 of \cite{BT90} for a ``direct" proof) that an ordinary functor $P:\ca\to\cx$ is solid if, and only if, $P$ is right adjoint and there is a class $\ce$ of morphisms in $\ca$ such that
\begin{itemize}
\item[(A)] all adjunction co-units lie in $\ce$;
\item[(P)] the pushout of a morphism in $\ce$ along any morphism exists in $\ca$, and any such lies in $\ce$;
\item[(W)] the wide pushout (= co-intersection) of a (possibly large) family of morphisms in $\ce$ with common domain exists in $\ca$, and any such lies in $\ce$.
\end{itemize}
For any morphism class $\ce$, the category $\ca$ is said to be $\ce$-{\em cocomplete} if conditions (P) and (W) hold.
Note that (W) forces every morphism in $\ce$ to be an epimorphism in $\ca$ (see \cite{BT78, T79}). Hence, the class $\ce$ may be assumed to be a class of epimorphisms {\em a priori}. Furthermore then, if $\ca$ is $\ce$-cowellpowered, the consideration of small ( = set-indexed) families in (W) suffices.
\end{remark}

Following the proof for the ordinary characterization theorem of Remark \ref{remark E-char} as given in Theorem 1.2 of \cite{BT90}, we easily arrive at the following characterization for strongly  order-solid functors, which entails the ordinary version as the discretely ordered case.

\begin{theorem}\label{theorem E-char}
An ordered functor $P:\ca\to \cx$ is strongly  order-solid if, and only if, $P$ is order-right adjoint, and there exists a class $\ce$ of order-epimorphisms in $\ca$ such that the (ordinary) conditions {\em (A), (P), (W)} hold.
\end{theorem}

\begin{proof} If $P$ is strongly order-solid, $P$ is order-right adjoint. Like in the proof for the ordinary case (see Theorem 2.1 of \cite{BT90}) one considers the class
$\ce$ of all those morphisms $e:A\to B$ in $\ca$ for which $Pe: PA\to PB$ is part of a universal $P$-extension of some family $\xi:PD\to PA$. But here, being order-$P$-epimorphic, such extension will make $e$ order-epic, {\em i.e.}, $r\cdot e\leq s\cdot e$ always implies $r\leq s$. Hence, $\ce$ is a class of order-epimorphisms which, being chosen as in the ordinary case, satisfies conditions (P) and (W). Furthermore, for the adjunction $F\dashv P$ with unit $\eta$ and co-unit $\varepsilon$, as in the ordinary case one has that, for every object $A$ in $\ca$,  $P\varepsilon_A:PFPA\to PA$ serves as a universal $P$-extension (of the pair $(\eta_{PA}: PA\to PFPA, 1_{PFPA}))$; but here we have to confirm that $P\varepsilon_A$ is order-$P$-epic. Indeed, since $P\varepsilon_A\cdot \eta_{PA}=1_{PA}$ and
$P$ is order-faithful by Proposition \ref{prop order ff}, for all $r,s:A\to B$ with $Pr\cdot P\varepsilon_A\leq Ps\cdot P\varepsilon_A$ one obtains $Pr\leq Ps$ and then $r\leq s$.
Consequently, $\varepsilon_A\in\ce$, which shows (A).

Conversely, we know that conditions (A), (P), (W) make $P$ solid as an ordinary functor, with universal $P$-extensions $(\alpha, A, q=Pe\cdot \eta_X:X\to PA)$ constructed in such a way that $e:FX\to A$ lies in the class $\ce$ (see Theorem 2.1 of \cite{BT90}). As $\eta_X$ is order-$P$-epic and $e$ is order-epic, $q$ must be order-$P$-epic, making it part of a strongly order-universal $P$-extension.
\end{proof}

\section{Examples of strongly  order-solid functors}\label{sec examples}

For many of our examples it is convenient to first consider them in a preorder-enriched context, so that $\Pos$ gets replaced by the larger Cartesian closed category $\Ord$ of preordered sets. We will freely use the terms introduced in Section \ref{sec strongly solid} in this context and thus talk about {\em preordered categories} and {\em functors},
 {\em strongly preorder-universal $P$-extensions} and {\em strongly preorder-solid} functors, 
 as well as about {\em preorder-$P$-epic} morphisms and {\em preorder-faithful} functors, keeping in mind that the latter two notions will no longer automatically imply that the morphisms will be  $P$-epic or the functors be faithful in the ordinary sense.

The following proposition turns out to be useful in many concrete situations.

\begin{proposition}\label{square}
In the commutative diagram
$$\xymatrix{\ca
\ar@{^(->}[r]^{H}\ar[d]_{P}\ar@{}[rd]|{\rm{}} & \ca'
\ar[d]^{P'}\\\cx\ar@{^(->}[r]_{J} & \cx'}$$
of preordered functors, let $H$ and $J$ be full emdeddings, with $H$ preorder-right adjoint. If $P'$ is strongly preorder-solid, then $P$ is also strongly preorder-solid, and trivially even strongly order-solid when it is an ordered functor.
\end{proposition}

\begin{proof}
Being preorder-right adjoint, $H$ is strongly  preorder-solid, and so is its composite with the strongly  preorder-solid functor $P'$ (see Remark \ref{first rem}). Quite trivially now, as $JP$ is strongly  preorder-solid, with $J$ being fully faithful, also $P$ is strongly  preorder-solid. 
Explicitly then, one constructs a strongly preorder-universal $P$-extension $(\alpha,A,q)$ of a $P$-cocone $\xi:PD\to X$ by composing a strongly  preorder-universal $P'$-extension $(\alpha',A',q')$ of $J\xi:P'HD\to JX$ with (the $P'$-image of) a reflection $r:A'\to HA$ into $\ca$:
$$\xymatrix{JPD\ar[r]^{J\xi}&
JX\ar[dr]_{Jf}\ar[r]^{q'}\ar@<1ex>@/^15.pt/[rr]^{Jq}&
P'A'\ar@{-->}[d]^{P't'}\ar[r]^{P'r}&P'HA\!=\!JPA\ar@{.>}[dl]^{JPt}
\\
&&P'\!HB=JPB&}$$ 
\end{proof}

\begin{example}\label{exa-top0} 
The functor $$S:\Top_0\to \Pos$$ provides the underlying set of a T$_0$-topological space $A$ with the (dual of the) {\em specialization order}, so that $x\leq y$ in $SA$ means that the neighbourhood filter of $x$ is finer than that of $y$ (or that the ultrafilter fixed at $x$ converges to $y$). 
With $S$, the category $\Top_0$ becomes order-enriched, that is: $f\leq g:A\to B$ in $\Top_0$ means $f\leq g:SA\to SB$ in $\Pos$, or $f(x)\leq g(x)$ in $SB$ for all $x\in A$. We show that $S$ is strongly order-solid.

In fact, since the specialization {\em pre}order may be defined for all topological spaces, so that $S$ is the restriction of a preordered functor $S'$ as in the diagram
$$\xymatrix{{\Top_0}
\ar@{^(->}[r]^{}\ar[d]_{S}\ar@{}[rd]|{\rm{}} & \Top
\ar[d]^{S'}\\\Pos\ar@{^(->}[r]_{} & \Ord\;,}$$
and since $\Top_0$ is epireflective in $\Top$, so that the surjective reflection morphisms make the embedding order-right adjoint, by Proposition \ref{square} it suffices to show that $S'$ is strongly preorder-solid.

Indeed, given a preordered set $(X,\leq)$ and any family of monotone maps $\xi_i: SD_i\to (X,\leq)$ defined on topological spaces $D_i,\,i\in I,$ we obtain a topology $\tau$ on the set $X$ by declaring open all those down-closed sets $U\subseteq X$ for which the set $\xi_i^{-1}(U)$ is open in $D_i$, for every $i\in I$. Then, obviously, 
$\id_X:(X,\leq)\to S(X,\tau)$ is monotone and preorder-$S$-epic, and all maps $\xi_i:D_i\to (X,\tau)$ are continuous. 
When we are given any monotone map $f:(X,\leq)\to SB$ with a topological space $B$, such that all maps $f\cdot \xi_i: D_i\to B$ are continuous, then $f^{-1}(V)$ is down-closed for every open set $V$ of $B$ and indeed open in $(X,\tau)$, thus making  $f:(X,\tau)\to B$ continuous.

As a particular consequence of $S$ being strongly order-solid, with Theorem \ref{strong vs order} and  Proposition \ref{colimit lift} below one concludes that $\Top_0$ has all (small-indexed) weighted limits and colimits (as described in Section \ref{sec order solid}) since $\Pos$ has them (see Examples \ref{examples again}(1)), as previously observed in \cite{CS11} and \cite{ASV15}. Likewise for $\Top$.
\end{example}

\begin{example}\label{exa-frm}
Every {\em frame} (= complete lattice in which the binary meet distributes over arbitrary joins) $A$ has an underlying {\em meet-semilattice} $UA$ which just forgets the existence of arbitrary joins; likewise, one may forget the information that a homomorphism of frames preserves arbitrary joins and just keep the information of preservation of finite meets, to obtain a functor $$U:\Frm\to\SLat.$$ With the order in both categories inherited from $\Pos$, this functor is order-enriched and right adjoint as such: for a meet-semilattice $X$, the adjunction unit $\downarrow:X\to UDX$ into the lattice $DX$ of down-closed subsets of $X$ (ordered by $\subseteq$) assigns to $x\in X$ the principal down-set $\downarrow x=\{z\in X\,|\,z\leq x\}$ in $X$; it is easily seen to be order-$U$-epic since every down-closed subset of $X$ is a join of principal down-sets.

In order to show that $U$ is strongly order-solid, we consider a meet-semilattice $X$ and a family of homomorphisms $\xi_i:UC_i\to X$, with frames $C_i,\, i\in I$. On the frame $DX$, one lets $\sim$ be the least congruence relation such that
$$\forall i\in I,\, K\subseteq C_i \,\big{(}\downarrow\xi_i(\bigvee K)\,\sim\,\bigcup_{a\in K}\downarrow\xi_i(a)\,\big{)}.$$
It is clear that, with the projection $p:DX\to A:=DX/\!\sim$, all maps $p\mathrel{\cdot\downarrow \cdot}\xi_i:C_i\to A$ become frame homomorphisms. Furthermore, any meet-semilattice homomorphism $f:X\to UB$ to a frame $B$, for which all maps $f\cdot \xi_i:C_i\to B$ are frame homomorphisms, gives us a frame homomorphism $f^{\sharp}:DX\to B$ whose induced congruence relation must contain $\sim$. Consequently, $f^{\sharp}$ factors as $f^{\sharp}=t\cdot p$ with a frame homomorphism $t:A\to B$. Since $q:=p\mathrel{\cdot\downarrow }\, : X\to UA$ is clearly order-$U$-epic, this shows that $q$ belongs to a strongly order-universal $U$-extension of the family $\xi$, as desired.
$$\xymatrix{UC_i\ar[r]^{\xi_i}
&
X\ar[r]^{\downarrow}\ar[dr]_f\ar@<1ex>@/^15.pt/[rr]^q
&
UDX\ar[r]^{p}\ar@{-->}[d]^{f^{\sharp}}
&
UA\ar@{..>}[dl]^{t}\\
&&UB&}$$

The above construction raises the question of how to ``compute" the least congruence relation $C$ on a frame $A$ containing a given relation $R$ on $A$ --- even though an answer is actually not needed in the proof above. In any case, the reader may consult \cite{PP} to see that 
the underlying set of $A/C$ may be taken to contain all elements of $A$ that are saturated with respect to $R$, that is: every $s\in A$  such that, for all $a,b,c\in A$,  $a\,R\,b$ implies $(a\wedge c\leq s\; \Longleftrightarrow\; b\wedge c\leq s)$. In this way, $A/C$ becomes a frame, with the map $\pi:A\to A/C$ that assigns to $x\in A$ the infimum of  all saturated elements $s$ with $x\leq s$, acting as the quotient map; $\pi$ satisfies
the condition $(a\,R\,b \Rightarrow \pi(a)=\pi(b))$ and is universal with respect to it, that is: any frame homomorphism $g:A\to B$ with $(a\,R\,b \Rightarrow g(a)=g(b))$ factors as $g=h\cdot \pi$ with a frame homomorphism $h:A/R\to B$.
\end{example}

\begin{example}\label{exa-slat} 
That also the forgetful functor
$$V:\SLat\to\Pos$$
 is strongly order-solid may be shown analogously to the previous example. Its left adjoint $E$ is described as follows: for an ordered set $X$, one takes $EX$ to contain the up-closures $\uparrow\! F$ of all finite subsets $F\subseteq X$, ordered by reverse inclusion $\supseteq$. Since $(\uparrow\! F)\cup(\uparrow\! G)=\uparrow\!(F\cup G)$, this makes $EX$ a meet-semilattice and the map $\uparrow:X\to VEX,\,x\mapsto \uparrow\! x,$ monotone and, in fact, as one easily sees, the unit of an adjunction, since $\uparrow\! F=\bigcup_{x\in F}\uparrow\! x$, {\em i.e.}, every element in $EX$ is a finite meet of ``generic" elements.

Given  a family of monotone maps $\xi_i:VC_i\to X$ with meet-semilattices $C_i\,(i\in I)$, one considers the least congruence relation $\sim$ on $EX$ satisfying the condition
$$\forall i\in I,\, a,b\in C_i\,\Big{(}\uparrow(\xi_i(a\wedge b))\sim\,\uparrow\xi_i(a)\cup\uparrow\xi_i(b)\;{\rm{and}}\;\uparrow\xi_i(\top_i)\sim\,\emptyset\,\Big{)},$$
where $\top_i$ denotes the top element in $C_i$. By definition of $\sim$, with the projection $p:EX\to A= EX/\!\sim$, one obtains meet-semilattice homomorphisms $p\mathrel{\cdot\uparrow\cdot}\xi_i:C_i\to A$ for all $i\in I$. Since $\sim$ is contained in the congruence relation induced by the canonical extension $f^{\sharp}:EX\to B$ of any monotone map $f:X\to VB$ to a meet-semilattice $B$ making all $f\cdot\xi_i$ homomorphisms, $f$ factors uniquely through $p\mathrel{\cdot\uparrow }\,: X\to VA$. That this map is order-$V$-epic follows again from the presentation $\uparrow\!F=\bigcup_{x\in F}\uparrow\!x$ of elements in $EX$.
\end{example}

As in Example \ref{exa-top0}, from Examples \ref{exa-frm} and  \ref{exa-slat} we can draw the conclusion that $\SLat$ and $\Frm$ have all (small-indexed) weighted (co)limits.

\begin{example}\label{exa-ordmon}
By an {\em ordered Abelian monoid} $A$ we understand a commutative monoid object in the category $\Pos$, that is: $A$ is a commutative monoid equipped with a partial order that makes its binary operation $+:A\times A\to A$ monotone. The morphisms of the resulting category ${\sf AbMon}(\Pos)$ are monotone monoid homomorphisms. With the order of the hom-sets of ${\sf AbMon}(\Pos)$ inherited from $\Pos$, we want to show that the forgetful functor
$W:{\sf AbMon}(\Pos)\to\Pos$
is strongly order-solid. For that, in consideration of the commutative diagram
$$\xymatrix{{\sf AbMon(\Pos)}
\ar@{^(->}[r]^{}\ar[d]_{W}\ar@{}[rd]|{\rm{}} & {{\sf AbMon}(\Ord)}
\ar[d]^{W'}\\\Pos\ar@{^(->}[r]_{} & \Ord\;,}$$
by Proposition \ref{square} it suffices to show that the forgetful functor $W':{\sf AbMon}(\Ord)\to\Ord$ of {\em preordered Abelian monoids} (which, in comparison to ordered Abelian monoids, are missing only the anti-symmetry) is strongly preorder-solid, and that the top-row full inclusion functor is preorder-right adjoint.  But the latter fact is easily guaranteed by General-Adjoint-Functor-Theorem-type arguments (see, for example, \cite{AHS, HST2014}), since ${\sf AbMon}(\Pos)$ is closed under point-separating families in ${\sf AbMon}(\Ord)$, so that we can focus on the former and first show that $W'$ is preorder-right adjoint.

To this end, since we are not aware of a proof presented in the specific situation considered here (see \cite{Fritz2017} and the literature cited in there), we rely on general principles to confirm that $W'$,
as an ordinary functor, is right adjoint, and apply the construction provided by Wyler's {\em Taut Lift Theorem}  \cite{Wyler71a}. Hence, for a preordered set $X$, we consider all monotone maps $f:X\to A_f$ whose codomain is any preordered Abelian monoid, and denote by $f^{\sharp}:FX\to A_f$ the homomorphism that extends $f$ to the free Abelian monoid $FX$ over the set $X$; it consists of all formal sums $\sum_{x\in X}n_xx$ (with non-negative integers $n_x$, all but finitely many being $0$), and $f^{\sharp}$ sends them to $\sum_{x\in X}n_xf(x)$. With
$$a\leq b:\iff \forall f:X\to A_f\,\big{(}f^{\sharp}(a)\leq f^{\sharp}(b)\,\big{)},$$
it is easy to see that $FX$ becomes a preordered Abelian monoid, making the insertion $\delta_X:X\to W'FX$ a $W'$-universal arrow, which turns out to be also order-$W'$-epic.

To finally see that $W'$ is strongly preorder-solid, given a family $\xi=(\xi_i:W'C_i\to X)_{i\in I}$ of monotone maps from preordered Abelian monoids $C_i\, (i\in I)$ to a preordered set $X$, we consider the least monoid congruence relation 
$\sim$ on $FX$ which, with the projection $p:FX\to FX/\!\sim$, makes all maps $p\cdot\delta_X\cdot\xi_i$ monoid homomorphisms. 
We must now define a preorder on $FX/\!\sim$, in such a way that $FX/\!\sim$ becomes a preordered Abelian monoid with monotone projection $p$. To this end, let us call a monotone map $f:X\to A_f$ $\xi$-admissible if $f\cdot \xi_i:C_i\to A_f$ is a monotone homomorphism for all $i\in I$ and then define, for all $a,b\in FX$,
$$p(a)\leq p(b):\iff \forall f:X\to A_f\;\, \xi{\text{-admissible}}\,\big{(}f^{\sharp}(a)\leq f^{\sharp}(b)\,\big{)}.$$
Since $\sim$ is generated by the pairs $(\delta_X(\xi_i(c+d)),\delta_X(\xi_i(c))+\delta_X(\xi_i(d))),\; c,d\in C_i,\,i\in I,$ one sees that this preorder is well defined and has the desired properties.

\end{example}

Categories of ordered algebras, of which ${\sf AbMon}(\Pos)$ is an example, have gained the attention of several authors; see, for instance, \cite{KV} and the references given there. Hence, in what follows, we extend the
previous example and consider any variety of any (possibly infinitary) type of general algebras instead of Abelian monoids. These are sets that come equipped with a class of (possibly infinitary) operations (instead of one binary and one nullary operation for monoids), which are required to satisfy certain equations (instead of the associativity, neutrality and commutativity requirements). Moreover, we must assume that one can form the free (pre)ordered general algebra of that type over a (pre)ordered set, with the insertion of generators being order-$(-)^1$-epic; here, as we explain next, $(-)^1$ denotes the forgetful functor from the category of (pre)ordered general algebras of the given type and their monotone homomorphisms to the category $\Pos$ (or $\Ord$). 

In the following theorem
we formulate these facts in terms of Lawvere-Linton (infinitary) algebraic theories (as originally introduced in \cite{Lawvere1963, Linton1966}; for a modern treatment in the finitary case, see \cite{ARV2011}). Explicitly then, paraphrasing \cite{Linton1966} in the spirit of \cite{ARV2011}, by an {\em (infinitary) algebraic theory} $\ct$ we mean a category whose class of objects is the class of cardinal numbers, such that every cardinal $n$ is the $n$-fold power of $1$ in $\ct$. An {\em ordered $\ct$-algebra} $A$ is a product-preserving functor $A:\ct\to\Pos$; its underlying ordered set is the value of $A$ at $1$.  When we denote the value of $A$ at $n$ more suggestively by $A^n$, then $A$ assigns to every $n$-ary term $t$ of $\ct$, {\em i.e.,} to every morphism $t:n\to 1$ in $\ct$, an $n$-ary monotone operation $At: A^n\to A^1$, written more conveniently as $t_A$.\footnote{For example, the morphisms $t:n\to m$ of the theory $\ct$ of Abelian monoids may be taken to be the homomorphisms $t:Fm\to Fn$ of the free Abelian monoids on $m$ and $n$ generators. Hence, for $m=1$, $t$ just picks an element in $Fn$, that is: an $n$-ary formal term, to which an algebra $A$ as defined here assigns the actual $n$-ary operation $t_A$ on its underlying set. Likewise for any other general algebraic structures admitting free algebras.}
A {\em monotone $\ct$-homomorphism} $f:A\to B$ of ordered $\ct$-algebras is simply a natural transformation; its underlying monotone map is the component of the transformation at 1, which must commute with the $n$-ary operations $t$; that is,  when we write the underlying map of $f$ as $f$ again,
 $f\cdot t_A=t_B\cdot f^n$. With the order on its hom-sets inherited from $\Pos$, this defines the ordered category ${\sf Alg}(\ct,\Pos)$, as a full subcategory of the ordered functor category $\Pos^{\ct}$. By replacing $\Pos$ by $\Ord$ one obtains the category of {\em preordered $\ct$-algebras} and the commutative diagram
$$\xymatrix{{\sf Alg(\ct,\Pos)}
\ar@{^(->}[r]^{}\ar[d]_{(-)^1}\ar@{}[rd]|{\rm{}} & {{\sf Alg}(\ct,\Ord)}
\ar[d]^{(-)^1}\\\Pos\ar@{^(->}[r]_{} & \Ord}$$ 
of preordered functors. We call $\ct$ {\em preorder-varietal} if the functor $(-)^1:{\sf Alg(\ct,\Ord)}\to\Ord$ is preorder-right adjoint and obtain, as in Example \ref{exa-ordmon}, the following quite general result:

\begin{theorem}\label{order varietal}
For every preorder varietal algebraic theory $\ct$, the forgetful functor
$$U^{\ct}:{\sf Alg}(\ct,\Pos)\to\Pos,\,A\mapsto A^1,$$
is strongly order-solid, and likewise when $\Pos$ is traded for $\Ord$.
\end{theorem}

We forgo the proof of the theorem, not only since it follows the same argumentation as that of Example \ref{exa-ordmon}, but also since the theorem is a special case of Theorem \ref{alg fun solid}, the proof of which we sketch in sufficient detail, albeit with a variation which avoids the use of Wyler's Theorem.

 While  Theorem \ref{order varietal} covers Examples \ref{exa-slat} and \ref{exa-ordmon}, a generalization of Example \ref{exa-frm} requires the consideration of {\em algebraic functors}, induced by morphisms of algebraic theories. Recall that a {\em morphism $K:\cs\to\ct$ of algebraic theories} $\cs,\ct$ is simply a functor that maps objects identically and preserves their status as direct products. For example, the embedding of the theory of meet-semilattices into the theory of frames is a morphism of algebraic theories. Any morphism $K$ of algebraic theories gives rise to the {\em ordered algebraic functor}
$${\sf Alg}(\ct,\Pos)\to{\sf Alg}(\cs,\Pos),\,A\mapsto AK,$$
which, for convenience, we denote by $K$ again. In the example just mentioned, this then is the forgetful functor $\Frm\to\SLat$ as considered in Example \ref{exa-frm}.

\begin{theorem}\label{alg fun solid}
The (pre)ordered algebraic functor induced by any morphism of preorder-varietal algebraic theories is strongly (pre)order-solid.
\end{theorem}

\begin{proof} (Sketch) As in Example \ref{exa-ordmon}, by Proposition \ref{square} it suffices that the (analogously defined) preordered functor $K'$ of the commutative diagram
$$\xymatrix{{\sf Alg(\ct,\Pos)}
\ar@{^(->}[r]^{}\ar[d]_{K}\ar@{}[rd]|{\rm{}} & {{\sf Alg}(\ct,\Ord)}
\ar[d]^{K'}\\{{\sf Alg}(\cs,\Pos)}\ar@{^(->}[r]_{} & {{\sf Alg}(\cs,\Ord)}}$$
is strongly preorder-solid.
With the notation for ordered algebras used also in the preordered case, the algebraic functor $K'$ commutes with the forgetful functors of the algebraic categories, that is: $U^{\cs}K'=U^{\ct}$. Assuming that both $U^{\ct}$ and $U^{\cs}$ are preorder-right adjoint, we first show that $K'$ is also preorder-right adjoint. To this end, we note that, according to Dubuc's {\em Adjoint Triangle Theorem} \cite{Dubuc}, the left adjoint $L$ of the ordinary functor $K'$ may be constructed with the help of the left adjoints $F^{\ct}\dashv U^{\ct},\;F^{\cs}\dashv U^{\cs}$ and their adjunction units $\eta$ and co-units $\varepsilon$.  An inspection of the proof of Dubuc's Theorem reveals that the unit $\kappa:1\to K'L$ of $L\dashv K'$ makes the diagram
$$\xymatrix{F^{\cs}U^{\cs}
\ar[r]^{\varepsilon^{\cs}}\ar[d]_{\mu U^{\cs}}\ar@{}[rd]|{\rm{}} & 1
\ar[d]^{\kappa}\\K'F^{\ct}U^{\cs}\ar[r]_{K'\pi} & K'L}$$
commute; here $\mu:F^{\cs}\to K'F^{\ct}$ is the mate of $\eta^{\ct}:1\to U^{\cs}(K'F^{\ct})=U^{\ct}F^{\ct}$, and $\pi: F^{\ct}U^{\cs}\to L$ is a (pointwise) regular epimorphism in ${\sf Alg}(\ct,\Ord)$. Now we can easily see that for every preordered $\cs$-algebra $B$, the unit $\kappa_B$ is preorder-$K'$-epic. Indeed, for $r,s: LB\to A$ in ${\sf Alg}(\ct,\Pos)$ with $K'r\cdot\kappa_B\leq K's\cdot\kappa_B$, the commutativity of the diagram gives $K'(r\cdot\pi_{B^1})\cdot\mu_{B^1}\leq K'(s\cdot\pi_{B^1})\cdot\mu_{B^1}$, where $B^1=U^{\cs}B$. But the mate  $\mu_{B^1}$
 of $\eta^{\ct}_{B^1}$ satisfies $U^{\cs}\mu_{B^1}\cdot\eta_{B^1}^{\cs}=\eta_{B^1}^{\ct}$, so that with $U^{\cs}K'=U^{\ct}$ we obtain
$$U^{\ct}(r\cdot\pi_{B^1})\cdot\eta_{B^1}^{\ct}\leq U^{\ct}(s\cdot\pi_{B^1})\cdot\eta_{B^1}^{\ct}.$$
Since $\eta_{B^1}^{\ct}$ is preorder-$U^{\ct}$-epic and $\pi_{B^1}$ surjective, $r\leq s$ follows.

Let $\xi=(\xi_i:K'C_i\to B)_{i\in I}$ be a family of monotone $\cs$-homomorphisms
from preordered $\ct$-algebras $C_i\,(i\in I)$ to a preordered $\cs$-algebra $B$. Largely neglecting to write down forgetful functors now, on the $\ct$-algebra $LB$ we consider the least congruence relation $\sim$ 
which makes the $\cs$-homomorphisms $p\cdot\kappa_B\cdot\xi_i$ \;$\ct$-homomorphisms, where $p:LB\to LB/\!\sim$ is the projection map. As in Example \ref{exa-ordmon}, we equip $LB/\!\sim$ with the preorder defined by
$$p(a)\leq p(b):\iff \forall f:B\to A_f\;\, \xi{\text{-admissible}}\,\big{(}f^{\sharp}(a)\leq f^{\sharp}(b)\,\big{)},$$
for all $a,b\in LB$; here $f$ runs through all monotone $\cs$-homomorphisms into some preordered $\ct$-algebra, and the $\xi$-admissibility of $f$ means that all maps $f\cdot\xi_i:C_i\to A$ need to be monotone $\ct$-homomorphisms; $f^{\sharp}:LB\to A$ denotes the $\ct$-homomorphism with $f^{\sharp}\cdot\kappa_B=f$. This makes $LB/\!\sim$ an object of ${\sf Alg}(\ct,\Ord)$ and $p\cdot\kappa_B$ belong to a preorder-universal $K$-extension of $\xi$.
\end{proof}

Theorem \ref{order varietal} appears as a special case of Theorem \ref{alg fun solid} when one chooses for $\cs$ the initial algebraic theory, given by the dual of the full subcategory of $\Set$ with object class all cardinal numbers, {\em i.e.} a skeleton of $\Set$.

We continue with an important example of a strongly order-solid functor of a category of a generalized type of ordered algebras which, however, is not covered by Theorem \ref{alg fun solid}, since only some of the algebraic operations are assumed to be monotone and, more importantly, since the order of the homomorphisms is not taken to be given pointwise by universal quantification over {\em all} elements of their common domain, but only over a part of it.

\begin{example}\label{OVec generating}
By an {\em ordered vector space} $V$ we understand a real vector space that comes equipped with a partial order for which the (binary) addition and all unary operations given by multiplication with any non-negative scalar $\lambda$ are monotone. Such $V$ defines the {\em positive cone} $PV=\{v\in V\,|\,v\geq 0\}$, and a linear map $f:V\to W$ is said to be {\em positive} 
if it maps $PV$ into $PW$; equivalently: if $f:V\to W$ is monotone. Given another positive linear map $g:V\to W$, one writes
$$f\leq g:\iff\forall v\geq0\;(f(v)\leq g(v)).$$
But to make sure that this preorder is anti-symmetric, we must assume that the positive cone $PV$ is {\em generating}, that is: $V=PV+(-PV)$. Hence,
we denote by $\OVec$ the category of all  ordered vector spaces $V$ whose positive cone is generating, and their positive linear maps. 
We obtain the ordered functor
$P:\OVec\to \Pos,$
and claim that $P$ is {\em strongly order-solid}.

As in Example \ref{exa-ordmon}, we use Proposition \ref{square}. We note that it suffices to show that the analogously defined functor $P'$ of {\em preordered vector spaces} with generating positive cones, which extends $P$ as in the diagram
$$\xymatrix{\OVec
\ar@{^(->}[r]^{}\ar[d]_{P}\ar@{}[rd]|{\rm{}} & \POVec
\ar[d]^{P'}\\\Pos\ar@{^(->}[r]_{} & \Ord\;,}$$
is strongly preorder-solid, since it is easy to see that any such preordered vector space $V$ admits a surjective reflection into $\OVec$:  just consider $V/U$, where $U=\{v\in V\, |\, v\leq 0\leq v\}$.

Proving first that $P'$ is preorder-right adjoint,
given a preordered set $X$, one extends its preorder $\leq$ and considers the least preorder $\leq$ of the free real vector space $FX$ with basis $X$ satisfying
\begin{itemize}
\item[1.] if $x\leq y$ in $X$, then $0\leq x\leq y$ in $FX$;
\item[2.] if $u\leq v$ in $FX$ and $w\in FX,\,\lambda\geq 0$, then $\lambda u+w\leq \lambda v+w$ in $FX$. 
\end{itemize}
In this way the positive cone of $FX$ becomes generating and the insertion $\eta_X:X\hookrightarrow FX$ a $P'$-universal arrow, which is also preorder-$P'$-epimorphic. 
Now, given a family of monotone maps $\xi_i:P'V_i\to X$ with preordered vector spaces $V_i,\,i\in I$, we consider all the vector space quotients  $q:FX\to FX/K_q$, where the preorder of $FX/K_q$ is such that it makes the quotient a preordered vector space, $q$ a positive map and all $g_i=q\ot \eta_X\ot \xi_i$  {\em positively linear}, so that $g_i(\lambda u+\mu w)=\lambda g_i(u) +\mu g_i(w)$ for all $v,\, w \in PV_i$ and $\lambda,\, \mu \geq 0$. For $K$ the intersection of all the subspaces $K_q$, the vector space $FX/K$ comes then equipped with the  preorder  given by
$$v+K\leq w+K\; :\iff \;\forall q\; (v+K_q\leq w+K_q).$$
This way we obtain a quotient $p:FX\to FX/K$ and monotone maps $p\ot \eta_X \ot \xi_i$. Since for each $i$, $PV_i$ is  generating in $V_i$, every $p\ot \eta_X \ot \xi_i$ has a unique linear extension $\alpha_i:V_i\to FX/K$ (because we may obtain a base contained in the positive cone).  The monotone map $p\ot \eta_X$ together with the family $\alpha$  forms the desired strongly preorder-universal $P'$-extension of $\xi$.
 We leave all details to the reader and refer to the literature, such as \cite{Jameson} or \cite{Schaefer}.
\end{example}

In the next example, we consider ordered vector spaces again, but now take the order of the hom-sets to be given by the pointwise order over the entire vector space, not just over the positive cone. Then the order becomes necessarily discrete  and, although we still obtain an order-faithful forgetful functor to $\Pos$ that is solid in the ordinary sense, it fails to be strongly order-solid. This shows in particular that, in Theorem \ref{thm characterization}, we cannot omit condition (b).

\begin{example}\label{OVec}Let $\OVec_{=}$ be the category of ordered real vector spaces and linear maps which preserve the order.  Given a pair of morphisms $f,g:V\to U$, since the inequality $f(u)\leq g(u)$ implies $g(-u)\leq f(-u)$, imposing the inequality $f(v)\leq g(v)$ for {\em all} $v\in V$ forces $f=g$. Hence, $\OVec_{=}$ is trivially order-enriched via the discrete order,  and the forgetful ordered functor
$$R:\OVec_{=}\to \Pos$$
is order-faithful. We show that $R$ is also solid in the ordinary sense but fails to be strongly order-solid.

In order to show that $R$ is solid, we can follow a path completely analogous to the one used in the previous example when we showed that $P$ is strongly order-solid, including  the use of Proposition \ref{square} in its non-enriched version (that is, with the discrete order between morphisms). Here, to show that the forgetful functor into $\Ord$ is a right-adjoint, we change condition 1 of the description of the preorder on the freely generated vector space $FX$ of Example \ref{OVec generating}, by replacing $0\leq x\leq y$ with $x\leq y$.

 As a solid ordinary functor, $R$ has a left adjoint $H$, but the adjunction units $X\hookrightarrow RHX$ will generally fail to be order-$R$-epic. Indeed for $X$ the 2-chain $\{a<b\}$, let the real valued maps $f^{\sharp}, \,g^{\sharp}:HX\to \mathbb{R}$ be determined by  the monotone maps $f,g:X\to \mathbb{R}$ with $0=f(a)< g(a)=f(b)=2<g(b)=3$. Then, for $u=b-a$, $f^{\sharp}(u)\not\leq g^{\sharp}(u)$. Consequently $R$ is not strongly order-solid.
 
In conclusion, the functor $R$  fulfils conditions (a) and (c) of  Theorem \ref{thm characterization}, that is, $R$  is ordinarily solid and order-faithful, but does not fulfil (b), since $R$  fails  to preserve inserters. Indeed, in  $\OVec_{=}$ inserters are just equalizers, since the order between morphisms is discrete, but not so in $\Pos$.

The above arguments also show that, analogously, we have a preorder-faithful functor which is solid but not strongly preorder-solid, and which does not preserve inserters.
\end{example}

\section{Order-enriched solid functors} \label{sec order solid}

Following Anghel's lead \cite{Anghel, Anghel1} we now look at notions of universal $P$-extension and solidity for order-enriched functors from the general enriched categorical perspective.
An {\em ordered diagram (of shape $\cd$)} in an ordered category $\ca$ is an ordered functor $D:\cd\to\ca$; we do not restrict the size of the ordered category $\cd$. For a given $\cd$ and an object $A$ in $\ca$, when there is no risk of confusion we denote the constant functor $\cd\to\ca$ with value $A$ again by $A$; a morphism $t:A\to B$ is then treated as a natural transformation of constant functors. A {\em weight} for an ordered diagram of shape $\cd$ is an ordered functor $W:\cd^{\rm op}\to\Ord$. (Note that, in forming $\cd^{\rm op}$, one turns around the arrows of $\cd$ while maintaining their order.) Every object $B$ in $\ca$ gives the weight 
$$\ca(D-,B)=\big(\cd^{\op}\xrightarrow{D^{\op}}\ca^{\op}\xrightarrow{\ca(-,B)}\Pos\big),\, i\mapsto \ca(Di,B),$$
and an $\ca$-morphism $t:A\to B$ then becomes a natural transformation $\ca(D-,t):\ca(D-,A)\to\ca(D-,B)$, {\em i.e.}, a morphism in the (potentially very large) ordered category $\Pos^{\cd^{\op}}$ of weights for $\cd$, the morphisms of which are ordered componentwise.
Pushing things even further, we note that, of course, $\ca(D-,B)$ is functorial in $B$, {\em i.e.}, one has the hom-functor 
$$H_D^W: \ca\to \Posl,\; B\mapsto \Pos^{\cd^{\op}}(W,\ca(D-,B)),$$
whose codomain may be very large\footnote{The objects of $\Posl$ are the partially ordered classes. Concerning the informal term ``very large" and the formation of $\Posl$, the same comment as the one made before Definition \ref{def strongly solid} (as footnote 5) applies here.}.

To fix our notation and terminology, we recall the notion of weighted colimit in both,  elementary "pointwise" form and standard terms of enriched category theory, before proceeding similarly for the enriched notion of order-solidity.

\begin{remarks}\label{def weighted colim}
 (1) A {\em weighted cocone} $\alpha:D\to A$ over an ordered diagram $D:\cd\to\ca$ of weight $W:\cd^{\op}\to\Pos$ (briefly referred to as a $W$-weighted cocone over $D$) is given by its {\em vertex} $A\in{\rm ob}\ca$ and a natural transformation $\alpha: W\to \ca(D-,A)$, that is: a family of $\ca$-morphisms
$\alpha_i^u:Di\to A\quad(i\in{\rm ob}\cd,\,u\in Wi),$ satisfying the conditions
\begin{itemize}
\item $\forall\, i\in{\rm ob}\cd,\, u,v\in Wi\; (\,u\leq v\Longrightarrow \alpha_i^u\leq\alpha_i^v\,)$;
\item $\forall\, d:i\to j$ in $\cd,\, v\in Wj \;(\,\alpha_j^v\cdot Dd=\alpha_i^{Wd(v)}\,)$.
\end{itemize}
$\cd$ is the {\em shape} of the cocone.

(2) A $W$-{\em weighted colimit of} $D$ is given by a $W$-weighted cocone $\alpha:D\to A$ such that 
\begin{itemize}
\item $\alpha: D\to A$ is universal amongst all $W$-weighted cocones $\beta:D\to B$, {\em i.e.}, any such $\beta$ factors through $\alpha$, so that there is a unique\footnote{Uniqueness is automatically guaranteed by the subsequent condition of $\alpha$ being order-epic.} $\ca$-morphism $t:A\to B$ with $\beta= t\cdot \alpha$; that is,\\
$\forall\, i\in{\rm ob}\cd,\, u\in Wi\;(\beta_i^u=t\cdot\alpha_i^u\,);$
\item $\alpha: D \to A$ is {\em order-epic}, so that for all $r,s:A\to B$ in $\ca$ one has the implication\\
$\forall\, i\in{\rm ob}\cd,\, u\in Wi\;(r\cdot\alpha_i^u\leq s\cdot\alpha_i^u)\Longrightarrow r\leq s;$\\\
we write more economically $(r\cdot\alpha\leq s\cdot\alpha\Longrightarrow r\leq s)$ for this implication.
\end{itemize}
It is easy to check that this equivalently means that $H_D^W$ is representable, {\em i.e.}, $H_D^W\cong \ca(A,-)$ as $\Posl$-valued functors, making $H_D^W$ in effect $\Pos$-valued.

(4) A weighted cocone $\alpha: D\to A$ whose weight $W= {\sf 1}:\cd^{\rm op}\to\Pos$ maps the $\cd$-objects constantly to the terminal ordered set $\sf 1$ is simply an ordinary cocone over the ordinary diagram $D_o$. Consequently, a $\sf 1$-weighted cocone $\alpha:D\to A$ is a weighted colimit precisely when it is an ordinary colimit of $D_o$ and order-epic. Such weighted colimits are usually called {\em conical}.

(5) A weighted cocone in $\ca$ over an empty diagram is just an object of $\ca$. By (4), a weighted colimit over the empty diagram is just an ordinary colimit, {\em i.e.,} an initial object of the category $\ca$. 

(6) A diagram $D$ in $\ca$ over the terminal (ordered) category $\cd=\sf 1$ can be viewed as an object $D$ of $\ca$; likewise, a weight $W$ with domain $\sf 1$ is to be considered as an ordered set $W$. A $W$-weighted cocone
with vertex $A\in{\rm ob}\ca$ is then a monotonely $W$-indexed family of morphisms $\alpha^u:D\to A$ in the ordered category $\ca$, so that $\alpha^u\leq\alpha^v$ whenever $u\leq v$ in $W$. If it is even a weighted colimit, $A$ is usually written as a {\em tensor product} $W\otimes D$, so that then the colimit property is described by the existence of order-isomorphisms
$$\ca(W\otimes D,B)\cong \Pos(W,\ca(D,B)),$$
naturally in $B\in{\rm ob}\ca$. The category $\ca$ is {\em tensored} if $W\otimes D$ exists for all $D\in{\rm ob}\ca$ and $W\in{\rm ob}\Pos$.

(7) Of particular interest is also the discretely ordered category $\cd=\{ a,b: 0\to 1\}$ with exactly two objects and exactly two non-identical arrows, together with the weight $W:\cd^{\rm op}\to\Pos$ defined by 
$$W1={\sf 1}=\{*\},\, W0={\sf 2}=\{u< v\},\,Wa(*)=u, \,Wb(*)=v,$$
sometimes referred to as the {\em Walking Two}. A diagram $D:\cd\to\ca$ is simply a pair $f,g: C\to B$ of morphisms in $\ca$, and a $W$-weighted colimit of that diagram is called a {\em co-inserter} for $f,g$, {\em i.e.,} it is an order-epic arrow $e:B\to A$, universal with respect to the property $e\cdot f\leq e\cdot g$, so that any $h:B\to E$ with $h\cdot f\leq h\cdot g$ factors as $h=t\cdot e$.

(8) {\em Weighted limits} in $\ca$ are, by definition, weighted colimits in $\ca^{\op}$.
\end{remarks}

\begin{definition}\label{def order solid}

(1) For an ordered functor $P:\ca\to\cx$ and a $W$-weighted cocone $\xi:PD\to X$ in $\cx$, we call the triple $(\alpha,A,q)$, consisting of a $W$-weighted cocone $\alpha:D\to A$ in $\ca$ and an $\cx$-morphism $q:X\to PA$, a {\em $P$-extension} of $\xi$ if
$q\cdot\xi=P\alpha$; that is, if
$q\cdot\xi_i^u=P\alpha_i^u$ for all $i\in\ob\cd, u\in Wi.$

(2) $(\alpha, A, q)$ is an {\em order-universal $P$-extension} of the $W$-weighted cocone $\xi:PD\to X$ if
\begin{itemize} 
\item[1.] $(\alpha, A, q)$ is a $P$-extension of $\xi$; 
\item[2.] $(\alpha, A, q)$ is {\em universal} with respect to property 1, that is: for every $P$-extension $(\beta,B,f)$ of $\xi$ there is a (unique) $\ca$-morphism $t:A\to B$ with $t\cdot\alpha=\beta$ and $Pt\cdot q=f$; 
\item[3.] $(\alpha,A,q)$ is {\em order-$P$-epi(morphi)c}, that is: for all $r,s:A\to B$ in $\ca$ one has the implication\\
$Pr\cdot e\leq Ps\cdot e,\; r\cdot \alpha\leq s\cdot \alpha\,\Longrightarrow\, r\leq s.$
\end{itemize}

(3) An ordered functor $P:\ca\to\cx$ is {\em order-solid} if every weighted cocone $\xi:PD\to X$ with any ordered diagram $D:\cd\to\ca$ of any weight $W:\cd^{\op}\to\Pos$ has an order-universal $P$-extension.
\end{definition}

\begin{remarks}\label{remarks order solid} 
(1) It is important to observe the difference of condition 3 in Definitions \ref{def strongly solid} and \ref{def order solid}: if the $\cx$-morphism $q:X\to PA$ with $A\in\ob\ca$ is order-P-epic, so is the $P$-extension $(\alpha,A,q)$, but not necessarily conversely. Reconciliation of this difference is the main aim of Theorem \ref{strong vs order} below, but the Open Problem \ref{oproblem 2} remains.

(2) A $P$-extension $(\alpha: D\to A,\,q:X\to PA)$ of a $1$-weighted cocone $\xi:PD\to X$ is an order-universal $P$-extension if it is a universal $P_o$-extension (previously called $P$-{\em semi-final}, see \cite{T79}). Consequently, {\em order-solid functors are solid in the ordinary sense}.

(3) 
For a $P$-extension $(\alpha, A, e:X\to PA)$
 of a weighted cocone $\xi:PD\to X$ over an empty diagram to be order-universal means more than having just a $P_o$-universal arrow at $X$ (in the sense of \cite{MacLane71}): in addition, the morphism $e:X\to PA$ (with the specified object $A\in{\rm ob}\ca$) needs to be order-$P$-epic and therefore serve as an adjunction unit in the enriched sense.
Consequently, {\em order-solid functors are order-right adjoint}.
 
(4) $(\alpha, A, q)$ is an {\em order-universal $P$-extension} of the $W$-weighted cocone $\xi:PD\to X$
if, and only if, the following diagram is a pullback, formally to be formed in the very large\footnote{See footnote 5.} category $\Posl$, even though its top row always lies in $\Pos$:

$$\begin{array}{rcl}&\xymatrix@C=4em{t\; \ar@{|->}[rr]&&\; Pt\; \ar@{|->}[rr]&&\;  Pt\cdot q}&\\
\xymatrix{t\ar@{|->}[d]\\t\cdot \alpha}& \xymatrix{\ca(A,B)\ar[d]\ar[r]^{P_{A,B}}&\cx(PA,PB)\ar[r]^{(-)\cdot q}&\cx(X,PB)\ar[d]\\ 
\Pos^{\cd^{\op}}(W,\ca(D-,B))\ar[rr]^{P^W_{D,B}}&&\Pos^{\cd^{\op}}(W,\cx(PD-,PB))}&\xymatrix{f\ar@{|->}[d]\\f\cdot \xi}\\
&\xymatrix@C=4em{\beta\; \; \ar@{|->}[rrrr]&&&&\; \; P\beta}&
\end{array}$$

(5) By (4) one has
$$\Pos^{\cd^{\op}}(W,\ca(D-,B))\times_{\Pos^{\cd^{\op}}(W,\cx(DP-,PB))} \cx(X,PB)\cong\ca(A,B)$$
in $\Posl$, naturally so with respect to $B$. Hence, considering the left-hand side as a functor $\ca\to\Posl$
in $B$, we see that the existence of an order-universal $P$-extension of $\xi$ is equivalent to the representability of that functor. A precursor of this statement for ordinary categories is contained in \cite{H72, H76}, and its generalization to the general enriched context in \cite{Anghel, Anghel1}.
\end{remarks}

It is easy to see that, for order-faithful functors, order-solidity is equivalent to strong order-solidity. For that let us first note:

\begin{lemma}\label{lemma epic}
For an order-faithful functor $P:\ca\to\cx$, if $(\alpha, A, q)$ is an order-universal $P$-extension of the weighted cocone $\xi:PD\to X$, then the morphism $q:X\to PA$ is $P$-epimorphic.
\end{lemma}
\begin{proof}
Assuming $Pr\cdot q\leq Ps\cdot q$ with $r,s:A\to B$, one has $P(\alpha\cdot r)=Pr\cdot q\cdot\xi\leq Ps\cdot q\cdot\xi=P(\alpha\cdot s)$ and then $r\cdot \alpha\leq s\cdot\alpha$ when $P$ is order-faithful. Since $(\alpha,A,q)$ is order-$P$-epic, $r\leq s$ follows.
\end{proof}

\begin{theorem}\label{strong vs order}
An ordered functor $P:\ca\to\cx$ is strongly  order-solid if, and only if, it is order-solid and order-faithful. In this case, if $\cx$ has all weighted limits of a given shape $\cd$, so does $\ca$.
\end{theorem}

\begin{proof}
For the first statement, if $P:\ca\to\cx$ is order-solid and order-faithful, the morphism $q$ of any order-universal $P$-extension of a weighted cocone $\xi$ is order-P-epic, by Lemma \ref{lemma epic}. This holds particularly when $\xi$ is a $\sf 1$-weighted cocone over a discrete diagram, as needed to satisfy Definition \ref{def strongly solid}. Conversely, let $P$ be strongly order-solid. By Proposition \ref{prop order ff}, P is order-faithful. Furthermore, in order to construct an order-universal $P$-extension of any $W$-weighted cocone $\xi:PD\to X$ with $D:\cd\to\ca$, one considers $\xi=(\xi_i^u:Di\to X)_{i\in\ob\cd,u\in Wi}$ as a discretely indexed family of morphisms, for which we have a strongly universal $P$-extension $(\alpha=(\alpha_i^u:Di\to A)_{i,u}, A, q)$, by hypothesis. Since $P$ is order-faithful, from $u\leq v$ in $Wi$ and, hence, $P\alpha_i^u=\xi_i^u\cdot q\leq \xi_i^v\cdot q=P\alpha_i^u$, one concludes $\alpha_i^u\leq \alpha_i^v$; and, analogously, for $d:i\to j$ in $\cd$, the faithfulness of $P$ implies $\alpha^v_j\ot Dd=\alpha_i^{Wd(v)}$. This makes $(\alpha,A,q)$ an order-universal $P$-extension of the $W$-weighted cocone $\xi$.

For the second statement (on the existence of weighted limits), one proceeds analogously to the proof of Theorem \ref{thm characterization} which deals with the special case of inserters. 
\end{proof}

In the ordinary case, that is, when the categories are ordered discretely, the notions of solid and strongly solid are equivalent, since, as proved in \cite{T79}, every solid functor is faithful. But we have not been able to decide whether order-faithfulness is an essential condition in  Theorem \ref{strong vs order}:

\begin{oproblem}\label{oproblem 2}
Is every order-solid functor order-faithful?
\end{oproblem}

\section{Order-solid functors and weighted colimits}\label{sec colim}

For the study of the behaviour of order-solid functors with respect to weighted colimits, we first consider  order-universal $P$-extensions of individual cocones, without the universal quantification over all such data. An easy, but nevertheless fundamental, observation in this regard is that, in generalization of a well-known property of the ordinary notions,  the order-universal $P$-extension of a weighted colimit of $PD$ gives a weighted colimit of $D$, as stated in the next proposition. In Remark \ref{rem cocomma} we recall some important types of weighted colimits and some of their properties.

\begin{proposition}\label{colimit lift}
 For an ordered functor $P:\ca\to\cx$, an ordered diagram $D:\cd\to\ca$ and a weight $W:\cd^{\rm op}\to\Ord$, let $\xi:PD\to  X$ be a $W$-weighted colimit of $PD$ in $\cx$. Then:
 
 {\em (1)} If $(\alpha:D\to A,\,A,\,q:X\to PA)$ is an order-universal $P$-extension of $\xi$, then $\alpha$ is a  $W$-weighted colimit of $D$ in $\ca$.
 
 {\em (2)} If $\alpha:D\to A$ is a $W$-weighted colimit of $D$ and $q:X\to PA$ the comparison morphism induced by $\xi$, then $(\alpha,A,q)$ is an order-universal $P$-extension of $\xi$.
\end{proposition}

\begin{proof}
(1) The colimit property of $\xi$ makes $P\beta$, for any cocone $\beta: D\to B$, factor as $P\beta= f\cdot\xi$, with $f:X\to PB$ in $\cx$. Order-universality of $(\alpha, A, q)$ gives $t:A\to B$ in $\ca$ with $Pt\cdot q=f$ and $t \cdot \alpha=\beta$. For any morphisms $r,s: A\to B$ with $r\cdot\alpha\leq s\cdot\alpha$, order preservation by $P$ gives $(Pr\cdot q)\cdot\xi\leq (Ps\cdot q)\cdot\xi$, whence $Pr\cdot q\leq Ps\cdot q$ follows since $\xi$ is order-epic. With $(\alpha, A, q)$ being order-$P$-epic, we conclude $r\leq s$. 

(2) The proof proceeds similarly to the proof of (1).
\end{proof}

\begin{corollary}\label{cor colim}
Let the ordered functor $P:\ca\to\cx$ admit order-universal $P$-extensions for all weighted cocones of shape $\cd$. Then, if $\cx$ has weighted colimits for diagrams of shape $\cd$, so does $\ca$; likewise for conical colimits instead of weighted colimits.
\end{corollary}

\begin{corollary}\label{new cor}
If $P:\ca\to\cx$ is order-solid and $\cx$ has weighted colimits of shape $\cd$, so does $\ca$.
\end{corollary}

\begin{corollary}\label{tensor coinserter}
 If the ordered functor $P:\ca\to\cx$ admits order-universal $P$-extensions for all weighted cocones of shape $\sf 1$ and $\cx$ is tensored, so is $\ca$.
\end{corollary}

\begin{remarks}\label{rem cocomma}
(1) Recall that, for morphisms $f:A\to B,\,g:A\to C$ in an ordered category $\ca$, a {\em cocomma object} for $(f,g)$ is given by an object $D$ and morphisms $p:B\to D,\,q:B\to D$ in $\ca$ with $p\cdot f\leq q\cdot g$
$$\xymatrix{A\ar[r]^g\ar[d]_f^{\; \; \; \; \rotatebox[origin=c]{45}{$\leq$}}&C\ar[d]^q\\
B\ar[r]_p&D}$$
 and $(p,q)$ is universal with that property (so that any pair $(k,l)$ with common codomain $E$ and $k\cdot f\leq l\cdot g$ must factor through $(p,q)$ by a morphism $t:D\to E$; moreover, the pair $(p,q)$ is required to be jointly order-epic.

(2) Similarly to co-inserters, also cocomma objects are easily recognized as weighted colimits: instead of the discretely ordered diagram shape with a parallel pair, consider a span and define the weight of its domain and codomains as for the Walking Two (see Remark \ref{def weighted colim}(7)).

(3) It is easy to see that one may construct the cocomma object of $(f,g)$ as in (1) by forming the conical coproduct $B+C$ with injections $i, j$ and then the co-inserter $c:B+C\to D$ of $(f\cdot i, g\cdot j)$. 

(4) Conversely to (3), having cocomma objects at one's disposal, one may construct the co-inserter $c:B\to C$ of a pair $(f,g: A\to B)$ by forming their cocomma object $(p,q: B\to D)$ and then the conical coequalizer $e: D\to C$ of the pair $(p,q)$.

(5) As a consequence of (3) and (4), in the presence of finite conical colimits, the existence of cocomma objects is equivalent to the existence of co-inserters.

(6) It is well known (see Lemma 3.13 of \cite{ASV15} in the dual situation) that the tensor product $W\otimes A$ may be constructed with conical copowers and co-inserters, as follows: presenting the order of $W$ as a subset $W_1$ of $W_0\times W_0$, with $W_0$ the underlying set of $W$, which comes with projections $d_1, d_2:W_1\to W_0$, one forms the conical copowers $W_1\otimes A$, $W_0\otimes A$ (which are, in fact tensor products with discretely ordered sets) and then the co-inserter of the induced morphisms $d_1\otimes A,\, d_2\otimes A: W_1\otimes A\to W_0\otimes A$.

(7) A standard result of enriched category theory (see Theorem 3.73 of \cite{Kelly82}) guarantees the existence of all weighted colimits of small shape in $\ca$ when $\ca$ has tensor products, conical coproducts and conical coequalizers. Taking the preceding remarks into account, one obtains: {\em The ordered category $\ca$ has all weighted colimits of small shape if, and only if, it has small-indexed conical coproducts, conical coequalizers and at least one --and then all-- of the following types of weighted colimits: tensor products, co-inserters, or cocomma objects.}
\end{remarks}
With Remark \ref{rem cocomma}(2), Proposition \ref{colimit lift} gives:
\begin{corollary}
If the ordered functor $P:\ca\to\cx$ admits order-universal $P$-extensions of cocones of finite shape and $\cx$ has cocomma objects (respectively, co-inserters), so does $\ca$.
\end{corollary}
We return to the examples presented in Section \ref{sec examples}.
\begin{examples}\label{examples again}
(1) Conical colimits in $\Pos$ are given by ordinary colimits. The tensor product $W\otimes A$ may be given as $W\times A$, ordered like the direct product. The cocomma object of $(f:A\to B,\,g:A\to C)$ has as its underlying set the union $B\cup C$, which may be assumed to be disjoint; one then maintains the orders of its subsets $B$ and $C$ and adds to that the condition that $y\leq z$ holds for $y\in B$ and $z\in C$ if $y\leq f(x)$ and $g(x)\leq z$ for some $x\in A$.

(2) In $\SLat$, the (conical) copower of $A$ indexed by a set $W_0$, denoted $A^{(W_0)}$, is the sub-semilattice of the power $A^{W_0}$ whose elements have all but finitely many coordinates equal to the  top element $\top$; each injection   $p_w$ maps every $a$ to $(a_u)\in A^{W_0}$ with $a_w=a$ and $a_u=\top$ for $u\not=w$.  The tensor product $A\otimes W$ is the quotient $A^{(W_0)}/\hspace*{-1.5mm}\sim\, $ where $W_0$ is the underlying set of $W$ and  $\sim$ is the least congruence containing the pairs $(p_u(a)\wedge p_v(a),\, p_u(a))$ for all $a\in A$ and $u\leq v$ in $W$. Given  $f:A\to B$ and $g:A\to C$ in $\SLat$, let  $B\times C$ be the  product in $\Pos$ (then also the conical product and conical coproduct in $\SLat$); the cocomma object of $(f,g)$  is the quotient $B\times C/\sim$, where $\sim$ is the least congruence relation with $(f(a),r)\sim (f(a), s)$ for all $a\in A$ and $r,s\geq g(a)$.

(3) The characterization of weighted colimits in $\Frm$ is more involved. Concerning the conical coproduct, if we first  take it in $\SLat$ and then form the  order-universal $U$-extension of the corresponding $U$-sink using the construction of  Example \ref{exa-frm}, we obtain precisely the description of the coproduct given in \cite{PP}. We can proceed in an analogous way for coequalizers (see also \cite{PP}), tensor products, co-inserters and cocomma objects.

(4) In $\OVec$, given  morphisms $f,g:V\to W$, we describe the co-inserter of $(f,g)$.  Let $C$ be the cone given by the sum of $PW$ with the cone   $S=\{g(v)-f(v)\,|\, v\in PV\}$, so $C=PW+S=\{u+w\,|\, u\in PW,\, w\in S\}$. The intersection $U=C\cap (-C)$ is a subspace of $W$. Let $W/U$ be the quotient space of $W$ whose order has positive cone $P(W/U)=\{w+U\,|\, w\in C\}$. Then the  co-inserter of the pair $(f,g)$ is precisely the projection $W\to W/U$. With this charaterization of the co-inserters, it is easy to obtain similar descriptions for tensor products and cocomma objects, using the fact that conical coproducts in $\OVec$ are just the usual direct sums of spaces with the positive cone given by the sum of the positive cones of the components of the sum; see Remark \ref{rem cocomma} (3) and (6).
\end{examples}

Guided by Anghel's Theorem 2.2.8 in \cite{Anghel1}, we now give a step-by-step analysis of what may be needed to construct an order-universal $P$-extension $(\alpha: D\to A,\,A,\, q: X\to PA)$ of a given $W$-weighted cocone $\xi:PD\to X$ with $D:\cd\to\ca$, assuming that we have some particular weighted colimits and a certain order-universal $P$-extension over an $\ob\cd$-indexed discrete diagram at our disposal.

{\bf{\em Step 1:}} For every $i\in{\rm ob}\cd$, we assume that the tensor products $(\lambda_i^u:Di\to Wi\otimes Di)_{u\in Wi}$ and $(\kappa_i^u:PDi\to Wi\otimes PDi)_{u\in Wi}$ 
with comparison morphisms $c_i: Wi\otimes PDi\to P(Wi\otimes Di)$ exist in $\ca$ and $\cx$, respectively. (Of course, by Corollary \ref{tensor coinserter}, the former tensor product may be obtained from the latter by an order-universal $P$-extension.) 
For every $i$, we let $\xi_i: Wi\otimes PDi\to X$ be the induced $\cx$-morphism satisfying $\xi_i\cdot\kappa_i^u=\xi_i^u$ for all $u\in Wi$. 

{\bf{\em Step 2:}} We assume that in $\cx$ there exists the conical {\em generalized pushout} diagram
$$\xymatrix{Wi\otimes PDi
\ar[r]^{c_i}\ar[d]_{\xi_i}\ar@{}[rd]|{\rm{}} & P(Wi\otimes Di)
\ar[d]^{\overline{\xi_i}}\\X\ar[r]_{p} & \overline{X}}$$
(which, of course, one may construct by first forming the conical pushout $(\xi'_i, c_i')$ of each pair $(\xi_i,c_i)$ and then the conical wide pushout  (= co-intersection) of $(c_i')_{i\in{\rm ob}\cd}$).

{\bf{\em Step 3:}} We assume that the (discretely) ${\rm ob}\cd$-indexed and $\sf 1$-weighted cocone $(\overline{\xi_i}:P(Wi\otimes Di)\to
\overline{X})_{i\in{\rm ob}\cd}$ has an order-universal $P$-extension $(\alpha_i:Wi\otimes Di\to A)_{i\in{\rm ob}\cd},\, A,\, q_o:\overline{X}\to PA)$.

We set $q:=q_o\cdot p:X\to PA$ and  $\alpha_i^u:=\alpha_i\cdot \lambda_i^u:Di\to A$ for all $i\in{\rm ob}\cd, u\in Wi$, and prove:

\begin{proposition}\label{construction}
Under the assumptions of Steps 1-3, and when (the ordinary functor) $P$ is faithful, one obtains a $W$-weighted cocone $\alpha=(\alpha_i^u)_{i,u}$ which, together with $q$, gives an order-universal $P$-extension of the given $W$-weighted cocone $\xi$.
\end{proposition}

\begin{proof}
Clearly, from $\lambda_i^u\leq\lambda_i^{u'}$ one obtains $\alpha_i^u\leq\alpha_i^{u'}$ for all $u\leq u'\in Wi,\,i\in{\rm ob}\cd$. Also, from $\xi_j^v\cdot PDd=\xi_i^{Wd(v)}$ one obtains $P(\alpha_j^v\cdot Dd)=P(\alpha_i^{Wd(v)})$ and then, when (the ordinary functor) $P$ is faithful, $\alpha_j^v\cdot Dd=\alpha_i^{Wd(v)}$
for all $d:i\to j$ in $\cd$ and $v\in Wj$. Hence, $\alpha$ is a $W$-weighted cocone, obviously satisfying $P\alpha=q\cdot \xi$.

Given any $W$-weighted cocone $\beta: D\to B$ in $\ca$ and an $\cx$-morphism $f:X\to PB$ with $P\beta= f\cdot\xi$, we consider the ${\rm ob}\cd$-indexed cocone  $(\beta_i)_i$ with $\beta_i\cdot\lambda_i^u=\beta_i^u$ for all $u\in Wi,\,i\in {\rm ob}\cd$. Then, from $P\beta_i\cdot c_i\cdot\kappa_i^u=f\cdot\xi_i\cdot\kappa_i^u$ for all $u$ one derives $P\beta_i\cdot c_i=f\cdot \xi_i$ for all $i$. The generalized pushout now gives an $\cx$-morphism $g:\overline{X}\to PB$ with $g\cdot p=f$ and $g\cdot \xi_i=P\beta_i$ for all $i$. Order-universality of the discrete cocone 
$(\alpha_i)_i$ together with $q_o$ finally produces an $\ca$-morphism $t:A\to B$ with $Pt\cdot q_o=g$ and $t\cdot\alpha_i=\beta_i$ for all $i\in{\rm ob}\cd$, from which one easily deduces $Pt\cdot q=f$ and $ t\cdot\alpha=\beta$.

To show that $(\alpha,A,q)$ is order-$P$-epic, we consider $\ca$-morphisms $r,s:A\to B$ with $Pr\cdot q\leq Ps\cdot q, \, r\cdot\alpha\leq s\cdot\alpha$. The latter inequality gives $r\cdot\alpha_i\leq s\cdot\alpha_i$ for every $i\in{\rm ob}\cd$ since the cocone $(\lambda_i^u)_u$ of the tensor product $Wi\otimes Di$ is order-epic, while the first inequality and the cocone $(\kappa_i^u)_u$ of the tensor product $Wi\otimes PDi$ being order-epic give $Pr\cdot q_o\cdot\overline{\xi_i}\leq Ps\cdot q_o\cdot\overline{\xi_i}$. Since also the conical generalized
pushout is order-epic, with $Pr\cdot q_o\cdot p\leq Ps\cdot q_o\cdot p$ one obtains
$Pr\cdot q_o\leq Ps\cdot q_o$. In conjunction with $r\cdot\alpha_i\leq s\cdot\alpha_i$ for every $i\in{\rm ob}\cd$ one can finally conclude $r\leq s$ since the order-universal $P$-extension $((\alpha_i)_{i\in{\rm ob}\cd},A,q_o)$ is order-$P$-epic.
\end{proof}

\begin{remarks}\label{second remarks}
(1) For any $A\in{\rm ob}\ca,\,W\in{\rm ob}\Pos$, such that the respective tensor products in $\ca$ and $\cx$ exist, we call the canonical morphism $c:W\otimes PA\to P(W\otimes A)$ a {\em tensor comparison morphism}.  In order to perform Step 2 it suffices that $\cx$ {\em has conical generalized pushouts of tensor comparison morphisms}; more precisely: 
the conical pushout of a tensor comparison morphism along any morphism exists in $\cx$, and the conical wide pushout of any family of such pushouts exists as well.

(2) If $\ca$ has tensor products preserved by $P$, then the needed tensor products and pushouts in $\cx$ as described in (1) trivially exist and are conical since then, by definition of preservation, all tensor comparison morphisms are isomorphisms, so that by putting $\overline{X}=X,\,\overline{\xi_i}=\xi_i\cdot c_i^{-1},\, p=1_X$ one obtains the needed generalized pushout diagram of Step 2.

(3) While, when tensor products exist in $\cx$, Corollary \ref{tensor coinserter} guarantees their existence also in $\ca$ if $P$ admits order-universal $P$-extensions of weighted cocones over diagrams of shape $\sf 1$, these will generally not be preserved by $P$, even when $P$ is strongly order-solid. For instance, let $A$ be the 2-chain ${\sf 2}=\{0<1\}$, and let $W$ be the discrete 2-element poset, thus $W\otimes A$ is just a conical copower. In $\SLat$, it is the diamond poset, but in $\Pos$ it is just the disjoint union of two copies of ${\sf 2}$. Hence, the strongly order-solid functor $V$ of Example \ref{exa-slat}  does not preserve tensor products.

(4) If one tightens the condition of Proposition \ref{construction} that $P$ be faithful to $P$ being order-faithful, then the construction leads us more generally from an (op)lax cocone $\xi$ to an (op)lax cocone $\alpha$, as the beginning of the proof of the Proposition shows. (Recall that an ({\em op}){\em lax $W$-weighted cocone} $\alpha: D\to A$ is given by an (op)lax natural transformation $\alpha: W\to\ca(D-,A)$, so that the identities $\alpha_j^v\cdot Dd=\alpha_i^{Wd(v)}$ of Remarks \ref{def weighted colim}(1) get traded for the inequalities $\alpha_j^v\cdot Dd\leq\alpha_i^{Wd(v)}$ (``$\geq$" in the op-lax case).) Consequently, with Propositions \ref{colimit lift}, \ref{construction}, the construction leads from ({\em op}){\em lax colimits} (= universal (op)lax cocones) in $\cx$ to (op)lax colimits in $\ca$ when $P$ is order-faithful.
\end{remarks}

With the Remarks \ref{second remarks} we obtain from Proposition \ref{construction} the following Corollary:

\begin{corollary}
For an ordered functor $P:\ca\to\cx$ and any ordered category $\cd$, all weighted cocones $\xi:PD\to X$ with $D:\cd\to \ca$ have order-universal $P$-extensions if

\begin{itemize}
\item[{\em (a)}] all $\sf1$-weighted and discrete $\ob\cd$-indexed cocones $(\xi_i:PD_i\to X)_i$ have order-universal $P$-extensions;
\item[{\em (b)}] $\ca$ has tensor products and $P$ preserves them;
\item[{\em (c)}] as an ordinary functor, $P$ is faithful.
\end{itemize}
\end{corollary}

We can now combine some of the previous statements and formulate a characterization of order-solid functors:
\begin{theorem}\label{thm tensored}
For the ordered functor $P:\ca\to\cx$, assume that $\cx$ has all tensor products $W\otimes PD$ (with $W$ in $\Pos$ and $D$ in $\ca$), as well as conical generalized pushouts of arbitrary families of tensor comparison morphisms. Then $P$ is order-solid if, and only if, $\ca$ is tensored and $P$ admits order-universal $P$-extensions for all $\sf 1$-weighted cocones of discrete shape. The assumption on $\cx$ is particularly satisfied when $\ca$ is tensored and $P$ preserves tensor products.
\end{theorem}

\begin{proof} When $P$ is order-solid, by Proposition \ref{colimit lift}(1), the existence of tensor products of the form $W\otimes PD$ in $\cx$ is sufficient to make $\ca$ tensored; also, trivially, the specified weighted cocones have order-universal $P$-extensions. Conversely, the existence of the specified order-universal $P$-extensions suffices to make the ordinary functor $P$ solid and, hence, faithful. With our assumptions on $\cx$, Proposition \ref{construction} now guarantees that $P$ is order-solid.

The additional claim follows from Remark \ref{second remarks}(2).
\end{proof}

\end{document}